\documentclass[12pt]{amsart}
\usepackage{amssymb}
\usepackage{graphicx}
\usepackage{color}
\usepackage{url}

\newtheorem{thm}{Theorem}
\newtheorem{lem}[thm]{Lemma}
\newtheorem{prop}[thm]{Proposition}
\newtheorem{example}[thm]{Example}
\newtheorem{cor}[thm]{Corollary}
\newtheorem{rem}[thm]{Remark}
\newtheorem{que}[thm]{Question}

\title{Positive braid knots of maximal topological~4-genus}
\author{Livio Liechti}
\thanks{The author is supported by the Swiss National Science Foundation ($\#159208$)}
\address{Mathematisches Institut, Universit\"at Bern, Sidlerstrasse 5, 3012 Bern, Schweiz}
\email{livio.liechti@math.unibe.ch}

\begin{document}
\begin{abstract}
We show that a positive braid knot has maximal topological 4-genus exactly if it has maximal signature invariant. 
As an application, we determine all positive braid knots with maximal topological 4-genus and compute the topological 4-genus for all positive braid knots with up to 12 crossings.
\end{abstract}
\maketitle

\section{Introduction}

The slice genus of a torus knot equals the ordinary genus $g$ by a theorem of Kronheimer and Mrowka~\cite{KM}. 
By work of Rudolph, this equality extends to the more general class of links bounding quasipositive surfaces, in particular to positive braid knots~\cite{Ru2}.
However, the story is very different for the \emph{topological 4-genus} $g_4$, 
i.e.~the minimal genus among surfaces which are properly, locally flatly embedded 
in the 4-ball and have a given knot $K$ as boundary (in contrast to the slice genus, where the embedding is required to be smooth).
A first example is due to Rudolph~\cite{Ru}: for the torus knot $T(5,6)$, the inequality $g_4<g$ holds. 
More recently, a large proportional difference $g - g_4$ with respect to $g$ was found for all torus knots with non-maximal signature $\sigma$ in~\cite{BFLL}.
On the other hand, there exists a lower bound, due to Kauffman and Taylor, for the topological 4-genus of knots: $2g_4(K) \ge \vert\sigma(K)\vert$ holds for any knot $K$~\cite{KT}.  
We show that for positive braid knots, this bound is in fact the only obstruction to non-maximal topological 4-genus, i.e.~$g_4<g$.

\begin{thm}
 \label{equivalenz}
For a positive braid knot $K$, the equality $g_4(K) = g(K)$ holds exactly if $\vert\sigma(K)\vert = 2g(K)$.
\end{thm}

Combining this result with Baader's classification of prime positive braid links of maximal signature~\cite{Ba}, 
we immediately get a full description of all prime positive braid knots of maximal topological 4-genus: they are exactly the torus knots of maximal signature.

\begin{cor}
\label{maximal}
The torus knots $T(2,n)$, $T(3,4)$ and $T(3,5)$ are the only prime positive braid knots $K$ with $g_4(K) = g(K)$. 
\end{cor}

Our proof of Theorem~\ref{equivalenz} uses two main ingredients. The first one is a homological criterion from~\cite{BFLL} using Freedman's disc theorem~\cite{Free}, 
allowing us to conclude $g_4<g$ for certain positive braids.
The second one is that \emph{genus defect} $\Delta g=g - g_4$ is inherited from \emph{surface minors}, i.e.~incompressible subsurfaces.
Similar to Baader's four surface minors $T$, $E$, $X$ and $Y$ obstructing maximal signature for positive braid links, 
we use enriched versions $\widetilde T$, $\widetilde E$, $\widetilde X$ and $\widetilde Y$ to obstruct maximal topological 4-genus for positive braid knots.

Theorem~\ref{equivalenz} also allows us to compute the topological 4-genus for positive braid knots $K$ with $\vert\sigma(K)\vert = 2g(K) -2$.
Combining the lower bound of Kauffman and Taylor with $g_4(K) < g(K)$ yields the exact result $g_4(K) = g(K)-1$.
This suffices to compute the topological 4-genus for prime positive braids knots with up to 12 crossings.
Table~1 lists all these knots, except for the torus knots $T(2,n)$, $T(3,4)$ and $T(3,5)$, 
which have maximal topological 4-genus. 
\begin{table}[h]
\begin{tabular}{| c | c | c | c | c |}
\hline
knot & braid notation & $g$ & $\vert\sigma\vert$ & $g_4$ \\ \hline
$10_{139}$ & $\sigma_1^4\sigma_2\sigma_1^3\sigma_2^2$ & 4 & 6 & 3 \\ \hline
$10_{152}$ & $\sigma_1^3\sigma_2^2\sigma_1^2\sigma_2^3 $ & 4 & 6 & 3  \\ \hline
$11n77$ & $\sigma_1^2\sigma_2^2\sigma_1\sigma_3\sigma_2^3\sigma_3^2$ & 4 & 6 & 3 \\ \hline
$12n242$ & $\sigma_1\sigma_2^2\sigma_1^2\sigma_2^7$ & 5 & 8 & 4 \\ \hline
$12n472$ & $\sigma_1\sigma_2^4\sigma_1^2\sigma_2^5$ & 5 & 8 & 4 \\ \hline
$12n574$ & $\sigma_1\sigma_2^6\sigma_1^2\sigma_2^3$ & 5 & 8 & 4 \\ \hline
$12n679$ & $\sigma_1^3\sigma_2^2\sigma_1^2\sigma_2^5$ & 5 & 8 & 4 \\ \hline
$12n688$ & $\sigma_1^3\sigma_2^4\sigma_1^2\sigma_2^3$ & 5 & 8 & 4 \\ \hline
$12n725$ & $\sigma_1\sigma_2^2\sigma_1^4\sigma_2^5$ & 5 & 8 & 4  \\ \hline
$12n888$ & $\sigma_1^3\sigma_2^3\sigma_1^3\sigma_2^3$ & 5 & 8 & 4  \\ \hline
\end{tabular}
\smallskip
\caption{Small positive braid knots.}
\end{table}
This list is created with the help of the software Knotinfo~\cite{Knotinfo}.
Previously, the values of the topological 4-genus for all these examples except $10_{152}$ were marked as unknown.
However, these values could also be deduced from work of Borodzik and Friedl on the algebraic unknotting number~\cite{BoFr}.
\newline

\textbf{Acknowledgements.} I warmly thank Sebastian Baader, Peter Feller and Lukas Lewark for many inspiring discussions and ideas that found their way into this article. 
Furthermore, I thank the referee for corrections and suggestions.

\section{Positive braids and trees}

A \emph{positive braid knot} is a knot that can be obtained from a positive braid via the closure operation,
an important example being torus knots.
A \emph{positive braid on $n+1$ strands} is a finite word in positive powers of the braid generators $\sigma_1, \dots ,\sigma_n$. 
By a theorem of Stallings, positive braid knots are fibred with the standard Seifert surface as fibre~\cite{St}.
As Baader did in~\cite{Ba}, we use \emph{brick diagrams} to visualise the fibre surface of positive braid knots: 
each horizontal bar corresponds to a braid generator $\sigma_i$ and each \emph{brick}, i.e. each rectangle, 
corresponds to a positive Hopf band in the plumbing construction of the fibre surface.
If two bricks link, it means that the core curves of the corresponding positive Hopf bands intersect once, see Figure~\ref{linking}.
Let the \emph{linking pattern} be the the plane graph obtained by putting a vertex into every brick and an edge between two vertices exactly if the corresponding bricks link.
It can be easily seen that if the intersection pattern of a positive braid $\beta$ is not connected, then the positive braid link $\widehat\beta$ is not prime. 
In fact, the converse is also true since positive braids are visually prime by a theorem of Cromwell~\cite{Cro}.
\begin{figure}
\def\svgwidth{90pt}
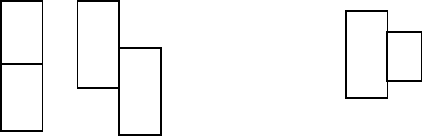
\caption{Bricks that link (on the left) and bricks that do not link (on the right).The two examples on the left yield the trefoil knot, while the example on the right yields a connected sum of two Hopf bands.}
\label{linking}
\end{figure}

\subsection{Trees} Let us for a moment consider the case where the linking pattern is a tree. 
There are many brick diagrams that yield the same tree as linking pattern. 
Since closures of positive braids corresponding to different brick diagrams might be equivalent as links in $\mathbf{R}^3$,
it is natural to ask whether the plane tree of the linking pattern uniquely determines the positive braid link up to ambient isotopy.
As we will see in the following remark, this is indeed the case.
\begin{rem}
\label{monodromiedeterminiert}
\emph{ 
The fibre surface $\Sigma(\beta)$ of a positive braid $\beta$ retracts to its brick diagram. 
Since for a successive plumbing of positive Hopf bands, the monodromy is conjugate to the product (in the succession of plumbing) of positive Dehn twists along 
the core curves of the Hopf bands~\cite{St},
the conjugacy class of the monodromy  
is completely determined by the plane tree given by the linking pattern of the brick diagram.
Therefore, also the corresponding fibred link $\widehat\beta$ is determined by the linking pattern of the brick diagram.
Indeed, the monodromy determines the mapping torus (up to homeomorphism fixing the boundary pointwise) 
and the fibredness condition dictates how to glue solid tori along the boundary of the mapping torus to obtain $\mathbf{S}^3$ containing a copy of the link $\widehat\beta$.
%
}
\end{rem}

Furthermore, if the linking pattern is a tree, a matrix for the Seifert form of the corresponding fibre surface $\Sigma=\Sigma(\beta)$ is particularly easy to describe: 
as a basis of $H_1(\Sigma; \mathbf{Z})$ take the core curves $[\alpha_i]$ of the positive Hopf bands corresponding to the bricks. 
A matrix $A$ for the Seifert form is then given by $A_{ii} = 1$ and $A_{ij} =1$ if $i<j$ and the curves $\alpha_i$ and $\alpha_j$ intersect 
(i.e.~if the corresponding vertices of the linking pattern are connected by an edge).
All other entries are equal to zero.

\begin{example}\emph{
Let $\widetilde T$, $\widetilde E$, $\widetilde X$ and $\widetilde Y$ be the canonical fibre surfaces
\begin{align*}
\widetilde T &= \Sigma(\sigma_1^5\sigma_2\sigma_1^4\sigma_2),\\
\widetilde E &= \Sigma(\sigma_1^7\sigma_2\sigma_1^3\sigma_2),\\
\widetilde X &= \Sigma(\sigma_1^2\sigma_2^2\sigma_1\sigma_3\sigma_2^2\sigma_3),\\
\widetilde Y &= \Sigma(\sigma_1^4\sigma_2^2\sigma_1^3\sigma_2),
\end{align*}
see Figure~\ref{tildas} for the corresponding brick diagrams and the linking patterns.
By exhibiting a two-dimensional subspace $B$ of $H_1(\widetilde X; \mathbf{Z})$ which is \emph{Alexander-trivial}, 
i.e.~ $\text{det}(A\vert_{B\times B}-t(A\vert_{B\times B})^{\top}) \in \mathbf{Z}[t^{\pm1}]$ is a unit for some matrix $A$ of the Seifert form,
it is shown in~\cite{BFLL} that the three-component link $\partial\widetilde X$ does not have maximal topological 4-genus. 
\begin{figure}[h]
\def\svgwidth{220pt}
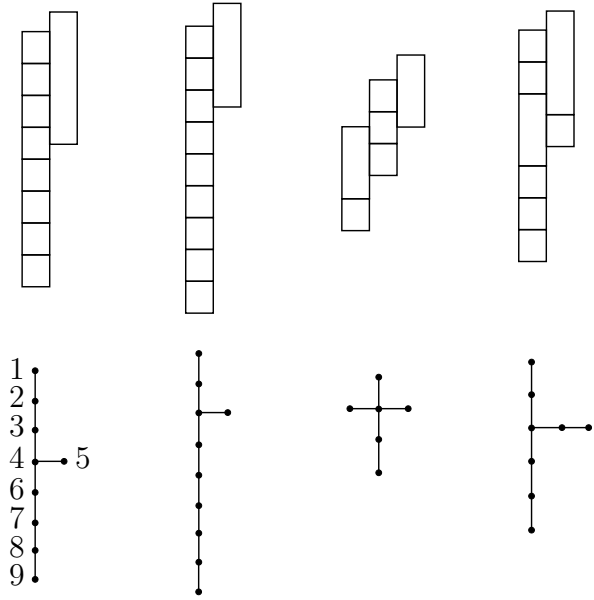
\caption{Brick diagrams for $\widetilde T$, $\widetilde E$, $\widetilde X$ and $\widetilde Y$ and the corresponding linking patterns.
The versions for $T$, $E$, $X$ and $Y$ are obtained by deleting the lowest brick and vertex, respectively.}
\label{tildas}
\end{figure}
More precisely, it is shown that the topological 4-genus equals one while the ordinary genus equals two.
In this example, we carry out the same computation for $\partial\widetilde T$, $\partial\widetilde E$ and $\partial\widetilde Y$.
For reasons of self-containedness, we also repeat the computation for $\partial\widetilde X$.
Number the vertices of the linking patterns in Figure~\ref{tildas} from top to bottom (and from left to right if several vertices are on the same level, as indicated for $\widetilde T$ in Figure~\ref{tildas}).
As a basis for the first homology, take the core curves of the corresponding Hopf bands with the chosen numbering. In this basis, consider the subspaces 
\begin{align*}
B_{\widetilde T} &=~\langle(-1,2,-3,4,-2,-3,2,-1,1)^{\top},~e_8\rangle,\\
B_{\widetilde E} &=~\langle(2,-4,6,-3,-5,4,-3,2,-1,1)^{\top},~e_9\rangle,\\
B_{\widetilde X} &=~\langle(-1,-1,2,-1,-1,0)^{\top},~e_6\rangle,\\
B_{\widetilde Y} &=~\langle(1,-2,3,-2,1,-2,1,-1)^{\top},~e_7\rangle
\end{align*}
of $H_1(\widetilde T; \mathbf{Z})$, $H_1(\widetilde E; \mathbf{Z})$, $H_1(\widetilde X; \mathbf{Z})$ and $H_1(\widetilde Y; \mathbf{Z})$, respectively.
Using the matrix $A$ of the Seifert form described above, it is a straightforward computation to see that in all four cases, the given subspaces are Alexander-trivial.
Writing $v$ for the first basis vector of $B_{\widetilde T}$ and $A_{\widetilde T}$ for the Seifert form corresponding to $\widetilde T$, one obtains 
\begin{align*}
v^\top A_{\widetilde T} v = 0,
v^\top A_{\widetilde T} e_8 = 1,\\
e_8^\top A_{\widetilde T} v = 0,
e_8^\top A_{\widetilde T} e_8 = 1,
\end{align*}
or, equivalently,
$$A_{\widetilde T}\vert_{B_{\widetilde T}\times B_{\widetilde T}}=
\begin{pmatrix}
  0 & 1\\
 0 & 1
 \end{pmatrix}.
$$
Consequently, $\text{det}(A_{\widetilde T}\vert_{B_{\widetilde T}\times B_{\widetilde T}}-t(A_{\widetilde T}\vert_{B_{\widetilde T}\times B})^{\top}) = t$, a unit in $\mathbf{Z}[t^{\pm1}]$.
The computation for the other cases works analogously.
Proposition~3 in~\cite{BFLL} now implies non-maximality of the topological 4-genus. Since the signature does not allow for a genus defect $g - g_4$ greater than one, 
we conclude 
\begin{align*}
g_4(\partial\widetilde T) &= g(\partial\widetilde T) -1 = 3,\\
g_4(\partial\widetilde E) &= g(\partial\widetilde E) -1 = 4,\\
g_4(\partial\widetilde X) &= g(\partial\widetilde X) -1 = 1,\\
g_4(\partial\widetilde Y) &= g(\partial\widetilde Y) -1 = 3.
\end{align*}
}\end{example}

In order to detect genus defect for a positive braid knot $\widehat\beta$, we search for minors $\widetilde{T},\widetilde{E},\widetilde{X}$ or $\widetilde{Y}$ in the fibre surface $\Sigma(\beta)$.
This is always based on the fact that the linking pattern of $\beta$ contains 
the tree corresponding to $\widetilde{T},\widetilde{E},\widetilde{X}$ or $\widetilde{Y}$ via deleting vertices and contracting edges.
One can then see that also $\Sigma(\beta)$ contains $\widetilde{T},\widetilde{E},\widetilde{X}$ or $\widetilde{Y}$, respectively, as a surface minor, implying $g_4(\widehat\beta) < g(\widehat\beta)$.
For example, Figure~\ref{H-tilde} shows how the tree corresponding to $\widetilde X$ is contained in the linking pattern of the positive braid $\sigma_1^2\sigma_2^3\sigma_1^2\sigma_2^2$.

\begin{figure}[h]
\def\svgwidth{210pt}
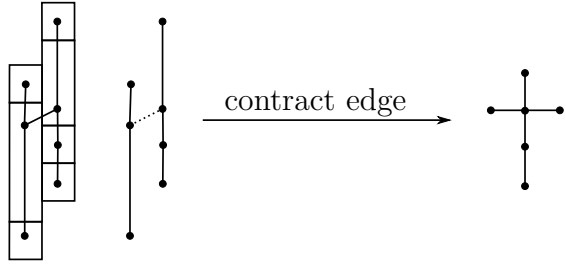
\caption{Contracting an edge of the intersection pattern of $\sigma_1^2\sigma_2^3\sigma_1^2\sigma_2^2$ yields the tree corresponding to $\widetilde X$.}
\label{H-tilde}
\end{figure}

Before we prove it for the case of positive braid knots, we show an analogue of Theorem~\ref{equivalenz} for knots obtained as \emph{plumbing of positive Hopf bands along a tree $\Gamma$}.
This notion generalises knots corresponding to brick diagrams having some plane tree as linking pattern. 
Starting from any finite plane tree $\Gamma$, we plumb positive Hopf bands (which are in one-to-one correspondence with the vertices of the tree) 
such that their core curves intersect once exactly if the corresponding vertices of $\Gamma$ are connected by an edge. 
Otherwise, they do not intersect.
Furthermore, they respect the circular ordering of the vertices given by the plane tree structure of $\Gamma$.
By the argument given in Remark~\ref{monodromiedeterminiert}, there is, up to ambient isotopy, only one way to do this. 
This construction is strictly more general than positive braid knots with a plane tree as linking pattern: vertices of a tree can have arbitrary valency, 
while for linking patterns associated with positive braid knots, this valency is bounded from above by $6$.

\begin{prop}
\label{trees_prop}
For a knot $K$ obtained by plumbing positive Hopf bands along a plane tree $\Gamma$, the equality $g_4(K) = g(K)$ holds exactly if $\vert\sigma(K)\vert = 2g(K)$.
\end{prop}

\begin{proof}
If $\vert\sigma(K)\vert = 2g(K)$, then $g_4(K) = g(K)$ follows from the signature bound of Kauffman and Taylor~\cite{KT}. If $\vert\sigma(K)\vert < 2g(K)$, we distinguish three different cases. 
If $\Gamma$ has at least three vertices of degree at least three, then the corresponding fibre surface contains $\widetilde{X}$ as a minor and thus $g_4(K) < g(K)$. 
If $\Gamma$ has two vertices of degree at least three, 
then at least one of the leaves has distance at least two from the closest vertex of degree at least three, since otherwise $K$ cannot be a knot.
Again the corresponding fibre surface contains $\widetilde{X}$ as a minor, since $\Gamma$ contains the tree corresponding to $\widetilde{X}$ via deleting vertices and contracting edges. 
If $\Gamma$ has only one vertex of degree at least three, then $\vert\sigma(K)\vert < 2g(K)$ holds if and only if $\Gamma$ contains the linking pattern of $T,E,X$ or $Y$ as an induced subgraph. 
This can be calculated directly from the associated Seifert forms. Alternatively, it also follows from Baader's classification of positive braid links of maximal signature~\cite{Ba}.
Again, for $K$ to be a knot, $\Gamma$ cannot be equal to $T,E,X$ or $Y$.
It follows that $\Gamma$ in fact contains the linking pattern of $\widetilde{T},\widetilde{E},\widetilde{X}$ or $\widetilde{Y}$ as an induced subgraph. 
Hence, the corresponding fibre surface contains $\widetilde{T},\widetilde{E},\widetilde{X}$ or $\widetilde{Y}$ as a surface minor.
\end{proof}

\section{Proof of Theorem 1}

The proof of Theorem~1 for positive braid knots $K$
is divided into two parts, depending on the \emph{positive braid index} of $K$, 
i.e.~the minimal index of a positive braid $\beta$ representing the knot $K$.
For $K$ of positive index at most three, we can essentially reduce the problem to Proposition~\ref{trees_prop}.
For $K$ of positive index at least four, we show that the strict inequality $g_4(K) < g(K)$ always holds.

\begin{prop}
\label{index3_prop}
For a knot $K$ obtained as the closure of a positive $3$-braid $\beta$, the equality $g_4(K) = g(K)$ holds exactly if $\vert\sigma(K)\vert = 2g(K)$.
\end{prop}

\begin{proof}
 We assume to have applied all possible braid relations $\sigma_1\sigma_2\sigma_1 \to \sigma_2\sigma_1\sigma_2$ to the braid $\beta$, so, up to cyclic permutation, $\beta$ can be assumed to be of the form
 $\sigma_1^{a_1}\sigma_2^{b_1}\cdots\sigma_1^{a_m}\sigma_2^{b_m}$, where $a_i>0$ and $b_i \ge 2$. 
 If $m\le2$, the linking pattern of the braid is a plane tree and we are done by Proposition~\ref{trees_prop}. 
 We now show that in the other cases we already have $g_4(\widehat\beta) < g(\widehat\beta)$.
 For this, let $m>2$ and remark that at least one of the $b_i$ has to be odd and hence at least three, otherwise the permutation given by the braid leaves the third strand invariant and $\widehat\beta$ is not a knot.
 
 \textit{Case 1: $m\ge4$}. Up to cyclic permutation, the braid $\beta$ contains the word $\sigma_1^2\sigma_2^3\sigma_1^2\sigma_2^2$ as a \emph{subword}, i.e. via reducing powers of occurrences of generators, 
 and thus the fibre surface of $\widehat\beta$ contains $\widetilde X$ as a minor, implying $g_4(\widehat\beta) < g(\widehat\beta)$.

 \textit{Case 2: $m=3$, $a_i=1$}. If, up to cyclic permutation, $(b_1,b_2,b_3)$ equals $(2,2,3)$, the Seifert form of $\widehat\beta$ is positive definite.
 If $(b_1,b_2,b_3)$ equals $(2,3,3)$ or $(3,3,3)$, the second strand is left invariant by the permutation given by the braid, so we can assume that one of the $b_i$ is at least four.
 Furthermore, since one of the $b_i$ has to be odd, $(b_1,b_2,b_3)$ can be assumed to be at least $(2,3,4)$ or $(5,2,2)$ with respect to the product order. 
 In both cases, $\beta$ contains the word $\sigma_1\sigma_2^5\sigma_1\sigma_2^4$ as a subword and thus the fibre surface of $\widehat\beta$ contains $\widetilde T$ as a minor, implying $g_4(\widehat\beta) < g(\widehat\beta)$.

 \textit{Case 3: $m=3$, at least one $a_i \ge 2$}. As before, one of the $b_i$ has to be at least three, say $b_1$. 
 If $a_1$ or $a_2$ is at least two, then $\beta$ contains, up to cyclic permutation, 
 $\sigma_1^2\sigma_2^3\sigma_1^2\sigma_2^2$ as a subword and thus the fibre surface of $\widehat\beta$ contains $\widetilde X$ as a minor.  
 Now assume $a_1=a_2=1$ and $a_3\ge 2$. We also assume $b_2=b_3=2$, otherwise we are, up to cyclic permutation, in the case we already dealt with. 
 Note that the permutation given by a braid of the form $\sigma_1\sigma_2^{b_1}\sigma_1\sigma_2^2\sigma_1^2\sigma_2^2$ leaves the second strand invariant, 
 so $a_3$ needs to be at least three in order for $\widehat\beta$ to be a knot. 
 Now, up to cyclic permutation, $\beta$ must contain the word $\sigma_2^5\sigma_1\sigma_2^2\sigma_1^3$ and the fibre surface of $\widehat\beta$ contains $\widetilde T$ as a minor, implying $g_4(\widehat\beta) < g(\widehat\beta)$.
\end{proof}

\begin{lem}
\label{minimalitylemma}
Let $\beta$ be a positive braid of index $\ge 3$. 
If for some $i$ the linking pattern of the subword of $\beta$ induced by the generators $\sigma_i$ and $\sigma_{i+1}$ is a path, then $\beta$ is not of minimal positive index.
\end{lem}

\begin{proof}
We can assume the subword of $\beta$ induced by the generators $\sigma_i$ and $\sigma_{i+1}$ to be $\sigma_i^k\sigma_{i+1}\sigma_i\sigma_{i+1}^l$, 
for some positive numbers $k$ and $l$.
This can be achieved by cyclic permutation and possibly reversing the order of the word $\beta$, operations that do not change the associated fibre surface.
Similarly, we can assume that all occurrences of generators with index smaller than $i$ come before the last occurrence of $\sigma_i$
and, likewise, all occurrences of generators with index greater than $i+1$ come after the first occurrence of $\sigma_{i+1}$. 
The situation is schematically depicted in Figure~\ref{nichtminimal} on the left.
\begin{figure}[h]
\def\svgwidth{150pt}
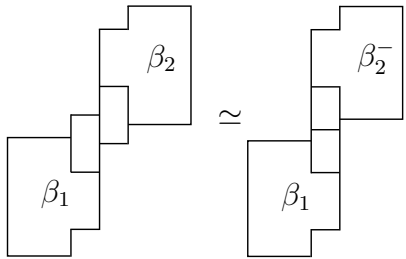
\caption{Two brick diagrams with the same linking pattern. The positive braid $\beta_2^-$ is defined to be $\beta_2$, but with all indices of braid generators decreased by one.}
\label{nichtminimal}
\end{figure}
Now consider the brick diagram
obtained by merging the two columns corresponding to the generators $\sigma_i$ and $\sigma_{i+1}$ as indicated in Figure~\ref{nichtminimal} on the right. 
By definition, the corresponding positive braid $\beta'$ has fewer strands than $\beta$.
To show that the closures of $\beta$ and $\beta'$ are ambient isotopic in $\mathbf{R}^3$, we study the corresponding monodromy homeomorphism of their fibre surfaces $\Sigma(\beta)$ and $\Sigma(\beta')$.
Since the linking patterns of the two brick diagrams are equal, the corresponding monodromies are conjugate and the closures of $\beta$ and $\beta'$ are ambient isotopic by the argument used in Remark~\ref{monodromiedeterminiert}.
\end{proof}

\begin{prop}
 \label{index4_prop}
 If $K$ is a prime knot obtained as the closure of a positive braid $\beta$ of minimal positive index $\ge 4$, then $g_4(K) < g(K)$.
\end{prop}

\begin{proof}
Let $\beta$ be a positive braid of minimal braid index $\ge4$ whose closure $\widehat\beta$ is a prime knot. 
We assume to have applied all possible braid relations $\sigma_i\sigma_{i+1}\sigma_i \to \sigma_{i+1}\sigma_i\sigma_{i+1}$ to $\beta$. 
This process terminates: it increases the sum of all indices of generators (counted with multiplicity) while not changing the number of generators.
In other words, the crossings of $\beta$ are as far to the right as possible.
We can furthermore assume that $\beta$ still contains, up to cyclic permutation, the subword $\sigma_1\sigma_2^2\sigma_1\sigma_2^2$, since otherwise $\beta$ would not be of minimal index. 

We first delete, without disconnecting the linking pattern, 
a minimal amount of occurrences of $\sigma_2$ so that the induced subword of $\beta$ in the first two generators is, after a possible cyclic permutation, of the form 
$\sigma_1^{a_1}\sigma_2^{b_1}\sigma_1^{a_2}\sigma_2^{b_2}$, where $b_1$ and $b_2$ are greater or equal to two.
For example, if the induced subword of $\beta$ in the first two generators is $\sigma_1\sigma_2^2\sigma_1\sigma_2^2\sigma_1\sigma_2^2$, we delete one occurrence of $\sigma_2$ (to the power two), 
yielding, after a possible cyclic permutation, $\sigma_1\sigma_2^2\sigma_1^2\sigma_2^2$.
Note that in case $a_1=a_2=1$, no generators $\sigma_2$ have to be deleted to achieve the desired form.

\textit{Case 1: $a_1=a_2=1,$ $b_1=b_2=2$}. In this case, we did not have to delete any occurrence of $\sigma_2$ 
and the induced subword of $\beta$ in the first two generators is exactly $\sigma_1\sigma_2^2\sigma_1\sigma_2^2$.
Both occurrences of $\sigma_2$ have to be split by an occurrence of $\sigma_3$, 
since otherwise the permutation given by $\beta$ would leave the first or second strand invariant, see Figure~\ref{invariantstrand}. 
\begin{figure}[h]
\def\svgwidth{75pt}
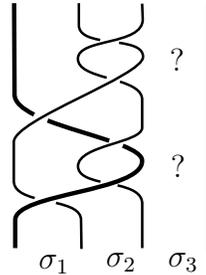
\caption{If no occurrence of $\sigma_3$ splits the first (second) occurrence of $\sigma_2$, the first (second) strand is left invariant by the permutation defined by $\beta$.}
\label{invariantstrand}
\end{figure}
Furthermore, these occurrences have to be to the power at least two, since we ruled out the possibility of a braid relation $\sigma_2\sigma_3\sigma_2 \to \sigma_3\sigma_2\sigma_3$. 
If one of the occurrences of $\sigma_3$ is to some power at least three,
$\beta$ contains, up to cyclic permutation, the subword $\sigma_2^2\sigma_3^3\sigma_2^2\sigma_3^2$ 
and thus the fibre surface of $\widehat\beta$ contains the minor $\widetilde{X}$, implying $g_4(\widehat{\beta}) < g(\widehat{\beta})$.
If the power of both occurrences of $\sigma_3$ is equal to two, we repeat the same argument: both occurrences of $\sigma_3$ have to be split by an occurrence of $\sigma_4$, 
otherwise the permutation given by $\beta$ would leave the first or second strand invariant. As before, we distinguish cases 
depending on the powers of the occurrences of $\sigma_4$. We repeat this argument and case distinction with increasing index as long as necessary. 
Eventually, some splitting occurrence has to be of power at least three and 
$\beta$ contains, up to cyclic permutation, the subword $\sigma_i^2\sigma_{i+1}^3\sigma_i^2\sigma_{i+1}^2$.

\textit{Case 2: $a_1=a_2=1,$ $b_1\ge3,$ $b_2=2$}. In this case, we did not have to delete any occurrence of $\sigma_2$
and the induced subword of $\beta$ in the first two generators is exactly $\sigma_1\sigma_2^{b_1}\sigma_1\sigma_2^2$.
As in Case~1, the second occurrence of $\sigma_2$ has to be split by an occurrence of $\sigma_3$
(otherwise the permutation given by $\beta$ would leave the second strand invariant),
so $\beta$ must contain a subword of the form $\sigma_1\sigma_2^{b_1}\sigma_1\sigma_2\sigma_3^{c_1}\sigma_2$. 
Note that $c_1$ must be greater or equal to two, since we applied all possible braid relations $\sigma_2\sigma_3\sigma_2\to\sigma_3\sigma_2\sigma_3$.
\begin{figure}[h]
\def\svgwidth{85pt}
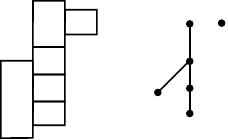
\caption{}
\label{nichtprim}
\end{figure}
Figure~\ref{nichtprim} depicts the brick diagram and intersecting pattern of this subword for $b_1=3$ and $c_1=2$. 
Since the intersection pattern is not connected, there has to be another occurrence of $\sigma_3$ in $\beta$, otherwise the closure $\widehat\beta$ would not be prime. 
What are the possibilities for the other occurrences of $\sigma_3$?
If the first occurrence of $\sigma_2$ is split by an occurrence of $\sigma_3$, again the occurrence of $\sigma_3$ has to be to the power at least two. 
Hence, $\beta$ contains, up to reversing order and cyclic permutation, the subword $\sigma_2^3\sigma_3^2\sigma_2^2\sigma_3^2$ 
and the fibre surface of $\widehat\beta$ contains the minor $\widetilde{X}$, implying $g_4(\widehat{\beta}) < g(\widehat{\beta})$.
Similarly, if $\beta$ contains, up to reversing order and cyclic permutation, the subword $\sigma_1\sigma_2^3\sigma_1\sigma_3\sigma_2^2\sigma_3$,
again the fibre surface of $\widehat\beta$ contains the minor $\widetilde{X}$, implying $g_4(\widehat{\beta}) < g(\widehat{\beta})$.
If we exclude these cases, the only two possibilities for the induced subword of $\beta$ in the first three generators are 
$\sigma_1\sigma_2^{b_1}\sigma_1\sigma_2\sigma_3^{c_1}\sigma_2\sigma_3^{c_2}$ and $\sigma_1\sigma_2^{b_1}\sigma_1\sigma_3^{c_2}\sigma_2\sigma_3^{c_1}\sigma_2$, 
which are, up to cyclic permutation, reverse to each other. If $c_2$ is greater or equal to two, 
the fibre surface of $\widehat\beta$ again contains the minor $\widetilde{X}$, implying $g_4(\widehat{\beta}) < g(\widehat{\beta})$, 
so we assume the induced subword of $\beta$ in the first three generators to be, up to reversing order and cyclic permutation, 
$\sigma_1\sigma_2^{b_1}\sigma_1\sigma_2\sigma_3^{c_1}\sigma_2\sigma_3$.
But in this case, $\beta$ restricted to the second and third generator has a path as linking pattern and is not minimal by Lemma~\ref{minimalitylemma}. 

\textit{Case 3: $a_1=a_2=1,$ $b_1,b_2\ge3$}.
The only possibility not considered in Case~2 is the following: $\beta$ contains, up to cyclic permutation, the subword $\sigma_1\sigma_2^3\sigma_1\sigma_2\sigma_3^2\sigma_2\sigma_3^2\sigma_2$, 
thus also $\sigma_1\sigma_2^3\sigma_1\sigma_2\sigma_3^2\sigma_2\sigma_3^2$ and the fibre surface of $\widehat\beta$ contains $\widetilde X$ as a minor, implying $g_4(\widehat{\beta}) < g(\widehat{\beta})$.
However, when reconsidering our discussion of Case~2, the powers of $\sigma_2$ appearing could be greater, 
so we get $\sigma_1\sigma_2^{b_1}\sigma_1\sigma_2^{b_2'}\sigma_3^{c_1}\sigma_2^{b_2''}\sigma_3^{c_2}$ as possibilities
for the induced subword of $\beta$ in the first three generators, where $b_2=b_2'+b_2''$. Again, note that if $c_2$ or $b_2''$ is greater or equal to two, 
the fibre surface of $\widehat\beta$ contains the minor $\widetilde{X}$, implying $g_4(\widehat{\beta}) < g(\widehat{\beta})$, 
so we assume the induced subword of $\beta$ in the first three generators to be, up to reversing order and cyclic permutation, 
$\sigma_1\sigma_2^{b_1}\sigma_1\sigma_2^{b_2'}\sigma_3^{c_1}\sigma_2\sigma_3$.
Again, $\beta$ restricted to the second and third generator has a path as linking pattern and is not minimal.

\textit{Case 4: $a_1+a_2\ge3,$ $b_1+b_2\ge5$}. We can apply the same arguments as in the cases above. 
From this it follows that if the fibre surface of $\widehat\beta$ contains no minor $\widetilde X$, then
the induced subword in the first three generators is, after the described process of deleting some generators $\sigma_2$, either
$\delta=\sigma_1^{a_1}\sigma_2^{b_1}\sigma_1^{a_2}\sigma_2^{b_2'}\sigma_3^{c_1}\sigma_2^{b_2''}\sigma_3^{c_2}$ or 
$\mu=\sigma_1^{a_1}\sigma_2^{b_1}\sigma_1^{a_2}\sigma_3^{c_2}\sigma_2^{b_2'}\sigma_3^{c_1}\sigma_2^{b_2''}$.  
As before, these two words are, up to cyclic permutation, reverse to each other. But since we might have deleted some generators $\sigma_2$ to obtain them, we should consider them separately. 
Again as before, if $c_2$ or $b_2''$ is greater or equal to two, 
the fibre surface of $\widehat\beta$ contains the minor $\widetilde{X}$, implying $g_4(\widehat{\beta}) < g(\widehat{\beta})$.
If we restrict $\delta=\sigma_1^{a_1}\sigma_2^{b_1}\sigma_1^{a_2}\sigma_2^{b_2'}\sigma_3^{c_1}\sigma_2\sigma_3$ to the second and third generator, 
the linking pattern is a path. Note that reinserting the deleted generators $\sigma_2$ would split $\sigma_1^{a_1}$ or $\sigma_1^{a_2}$. 
In any case, the linking pattern of $\beta$ restricted to the second and third generator is still a path and $\beta$ is not minimal. 
This does not necessarily hold for the other possibility $\mu=\sigma_1^{a_1}\sigma_2^{b_1}\sigma_1^{a_2}\sigma_3\sigma_2^{b_2'}\sigma_3^{c_1}\sigma_2$.
However, note that if ${b_2'}$ is greater or equal to two, then $\mu$ contains the subword $\sigma_1\sigma_2^2\sigma_1\sigma_3\sigma_2^2\sigma_3^2$
and the fibre surface of $\widehat\beta$ contains the minor $\widetilde{X}$, implying $g_4(\widehat{\beta}) < g(\widehat{\beta})$. 
So we are left with the possibility $\mu=\sigma_1^{a_1}\sigma_2^{b_1}\sigma_1^{a_2}\sigma_3\sigma_2\sigma_3^{c_1}\sigma_2$.
If all the deleted occurrences of $\sigma_2$ appeared before the first occurrence of $\sigma_3$ in $\mu$, after a cyclic permutation 
the linking pattern of $\beta$ restricted to the second and third generator again is a path and $\beta$ is not minimal.
If some deleted occurrence of $\sigma_2$ appeared after the first occurrence of $\sigma_3$ in $\mu$, then $\beta$ contains the word $\sigma_1\sigma_2^2\sigma_1\sigma_3\sigma_2^2\sigma_3^2$
as a subword and, as before, $g_4(\widehat{\beta}) < g(\widehat{\beta})$.

\textit{Case 5: $a_1+a_2\ge3,$ $b_1=b_2=2$}. In this case, there is one last new possibility: 
as in Case~1, the word $\sigma_2\sigma_3^{c_1}\sigma_2^2\sigma_3^{c_2}\sigma_2$ could be a subword of $\beta$ (without directly yielding $\sigma_2^3\sigma_3^2\sigma_2^2\sigma_3^2$ as a subword).
Again, since we applied all possible braid relations $\sigma_2\sigma_3\sigma_2 \to \sigma_3\sigma_2\sigma_3$, $c_1$ and $c_2$ are greater or equal to two.
If $\beta$ should, up to cyclic permutation, neither contain $\sigma_2^3\sigma_3^2\sigma_2^2\sigma_3^2$ nor $\sigma_2^2\sigma_3^3\sigma_2^2\sigma_3^2$ as subword, 
then $c_1$ and $c_2$ are both equal to two and the induced subword of $\beta$ in the first two generators is exactly $\sigma_1^{a_1}\sigma_2^2\sigma_1^{a_2}\sigma_2^2$.
In particular, we again did not have to delete any occurrence of $\sigma_2$ in the deletion process described above.
If the induced subword of $\beta$ in the first three generators was $\sigma_1^{a_1}\sigma_2\sigma_3^2\sigma_2\sigma_1^{a_2}\sigma_2\sigma_3^2\sigma_2$,
the permutation given by $\beta$ would leave the third strand invariant and $\widehat\beta$ would not be a knot. Thus, there has to be at least one more occurrence of a generator $\sigma_3$.
This gives the last two possibilities of induced subwords of $\beta$ in the first three generators: 
$\gamma=\sigma_1^{a_1}\sigma_2\sigma_3^2\sigma_2\sigma_1^{a_2}\sigma_3^{c_3}\sigma_2\sigma_3^2\sigma_2$ and 
$\sigma_1^{a_1}\sigma_2\sigma_3^2\sigma_2\sigma_1^{a_2}\sigma_2\sigma_3^2\sigma_2\sigma_3^{c_3}$,
which are, up to cyclic permutation, reverse to each other.
If $\beta$ is of index four, then actually $\beta$ would have to equal $\gamma$. But the closure $\widehat\gamma$ can never be a knot,
since the last two strands get permuted among themselves independently of $a_1,a_2$ and $c_3$. 
Now let $\beta$ be of index at least five.
If $\gamma$ is the induced subword of $\beta$ in the first three generators, 
one of the occurrences of $\sigma_3^2$ has to be separated by an occurrence of $\sigma_4$ to the power at least two 
(recall that we ruled out the possibilty of a braid relation $\sigma_3\sigma_4\sigma_3 \to \sigma_4\sigma_3\sigma_4$), 
since otherwise the first two strands would get permuted among themselves by $\beta$ and $\widehat\beta$ would not be a knot.
One can then see that $\beta$ contains, up to reversing order and cyclic permutation, one of the subwords
$\sigma_3^3\sigma_4^2\sigma_3^2\sigma_4^2$ or $\sigma_2\sigma_3\sigma_4^2\sigma_3\sigma_4\sigma_2^2\sigma_3^2\sigma_2$,
each of which guarantees the existence of a minor $\widetilde X$ in the fibre surface of $\widehat\beta$, implying $g_4(\widehat{\beta}) < g(\widehat{\beta})$. 
\end{proof}

\section{Characterisation by forbidden minors}

In the previous section, we established a characterisation of positive braid knots with maximal topological 4-genus by the forbidden surface minors $\widetilde{T},\widetilde{E},\widetilde{X}$ and $\widetilde{Y}$.
More precisely, we used all minors in the case of braid index 3, while we only needed the minor $\widetilde{X}$ in the case of minimal positive braid index $\ge 4$.
Another way to think of this is that for positive braid knots, genus defect $g-g_4 \le 0$ is characterised by these four forbidden minors.
A natural question to ask is whether a similar result holds for larger genus defect.

\begin{que}
\label{characterisation}
 For positive braid knots, can genus defect $g-g_4 \le c$ be characterised by finitely many forbidden surface minors for any $c\ge0$?
\end{que}

As noted by Baader and Dehornoy~\cite{BaDe}, this is implied by Higman's Lemma if we restrict ourselves to positive braids of index bounded by some natural number $n$.
A possible way of dealing with Question~\ref{characterisation} in the unbounded case could be to give a positive answer to the following question, 
which can be thought of as a strengthening of Proposition~\ref{index4_prop}.

\begin{que}
\label{linear}
For positive braid knots, is $g-g_4$ bounded from below by an increasing affine function of the positive braid index?
\end{que}

\end{document}